\documentclass[11pt,leqno]{amsart}

\usepackage{graphicx}
\usepackage{geometry}
\usepackage{amsmath}
\usepackage{amssymb}
\usepackage{amsthm}
\usepackage{amsfonts}
\usepackage{mathrsfs}
\usepackage{enumerate}
\usepackage[scriptsize,hang,raggedright]{subfigure}
\usepackage[cmyk]{xcolor}

\usepackage{hyperref}
\usepackage{pdfsync}
\usepackage{dsfont}
\usepackage{color}

\usepackage{cite}
\usepackage{bbm}
\usepackage{multirow}

\numberwithin{equation}{section} 

\newtheorem{theorem}{Theorem}[section]
\newtheorem{proposition}[theorem]{Proposition}
\newtheorem{lemma}[theorem]{Lemma}

\theoremstyle{remark}
\newtheorem{remark}[theorem]{Remark}
\theoremstyle{definition}
\newtheorem{example}[theorem]{Example}

\newcommand{\R}{\mathbb{R}}

\newcommand{\ba}{\begin{array}}
\newcommand{\ea}{\end{array}}

\newcommand{\bthm}{\begin{theorem}}
\newcommand{\ethm}{\end{theorem}}
\newcommand{\bprop}{\begin{proposition}}
\newcommand{\eprop}{\end{proposition}}
\newcommand{\blemma}{\begin{lemma}}
\newcommand{\elemma}{\end{lemma}}
\newcommand{\bexmpl}{\begin{example}}
\newcommand{\eexmpl}{\end{example}}

\newcommand{\beqn}{\begin{equation}}
\newcommand{\eeqn}{\end{equation}}
\newcommand{\beqns}{\begin{equation*}}
\newcommand{\eeqns}{\end{equation*}}

\newcommand{\supp}{\operatorname{supp}}

\newcommand{\pr}{\prime}
\newcommand{\pt}{\partial}
\newcommand{\arrow}{\rightarrow}

\newcommand{\Rn}{\mathbb{R}^n}

\newcommand{\ol}{\overline}

\newcommand{\Hn}{\mathcal{H}^{n-1}}
\newcommand{\Prb}{\mathscr{P}}

\renewcommand{\leq}{\leqslant}
\renewcommand{\geq}{\geqslant}

\newcommand{\Ttwo}{{\mathbb{T}^2}}
\newcommand{\Tn}{{\mathbb{T}^n}}
\newcommand{\clB}{\ol{B}_R(x_0)}
\newcommand{\B}{B_R(x_0)}

\definecolor{mygreen}{rgb}{0.1,0.75,0.2}

\newcommand{\eps}{\epsilon}

\newcommand{\E}{\mathsf{E}}

\newcommand{\Om}{\Omega}
\newcommand{\x}{\mathbf{x}}

\DeclareMathOperator{\inte}{int}

\DeclareMathOperator{\Per}{Per}
\DeclareMathOperator{\dive}{div}

\DeclareMathOperator*{\dist}{dist}

\title[Sharp Interface of Nanoparticle-Polymer Blend Model]{Sharp Interface Limit of an Energy 
Modelling Nanoparticle-Polymer Blends}
\author{Stan Alama}
\address{Department of Mathematics and Statistics, McMaster University, Hamilton, ON L8S 4K1, Canada}
\email{alama@mcmaster.ca}
\author{Lia Bronsard}
\address{Department of Mathematics and Statistics, McMaster University, Hamilton, ON L8S 4K1, Canada}
\email{bronsard@mcmaster.ca}
\author{Ihsan Topaloglu}
\address{Department of Mathematics and Statistics, McMaster University, Hamilton, ON L8S 4K1, Canada}
\email{itopalog@math.mcmaster.ca}
\date{\today}                                        

\begin{document}

\begin{abstract}
We identify the $\Gamma$-limit of a nanoparticle-polymer model as the number of particles goes to infinity and as the size of the particles and the phase transition thickness of the polymer phases approach zero. The limiting energy consists of two terms: the perimeter of the interface separating the phases and a penalization term related to the density distribution of the infinitely many small nanoparticles. We prove that local minimizers of the limiting energy admit regular phase boundaries and derive necessary conditions of local minimality via the first variation. Finally we discuss possible critical and minimizing patterns in two dimensions and how these patterns vary from global minimizers of the purely local isoperimetric problem. 
\end{abstract}
\maketitle

\section{Introduction}\label{sec:intro}

In many applications, engineering a self-regulating, stable structure with predetermined physical properties is highly desirable. Here we consider the case of block copolymers, for which a composite is created by adding solid ``filler'' nanoparticles in a blend of macromolecules to create high-performance polymers that are used in, for example, solid-state rechargeable batteries, photonic band gap devices, etc. (cf. \cite{Balazs_et_al} and references therein). Depending on the desired physical properties, block copolymers can be used to direct assembly of nanoparticles or, vice-versa, nanoparticles can be placed in the polymeric matrix to alter the morphology of block copolymer microdomains in both the strong and intermediate segregation limits (see e.g. \cite{Kim_et_al,Lee_et_al_2002,Lin_et_al,Thompson_et_al}). Nanoparticles are also used in altering morphologies of immiscible mixtures such as oil/water mixtures or fluid-bicontinuous gels with a microreaction medium and in controlling domains of coarsening (cf. \cite{Cui_et_al,Ginzburg_et_al_1999,Lin_et_al,Stratford_et_al}).

In this paper we consider an Ohta--Kawasaki-type model for a nanoparticle-block copolymer composite and study its limit as the interfacial length scale becomes small, within the framework of $\Gamma$-convergence.  The limiting sharp-interface model combines three effects via interfacial, repulsive nonlocal and bulk terms. 
As a first step, in this paper we concentrate on two of these and provide a more detailed analysis of a sharp-interface energy consisting only of the interfacial and bulk terms.  Clearly, retaining the repulsive interaction term would enrich the energy landscape of the problem and admit a wider variety of potential patterns for minimizers; however, by excluding nonlocal effects we may focus on the confining effect of the nanoparticles which already dramatically alters the morphology of the energy minimizers for the isoperimetric problem. We illustrate this with a specific example in two dimensions, in which the presence of nanoparticles influences the minimizing configuration to switch from a lamellar to a circular interface.

\subsection{The Limiting Problem}\label{subsec:problem} 
We assume that the nanoparticles have a given, fixed distribution, described by an absolutely continuous probability measure $\mu\in\Prb_{\text{ac}}(\Tn)$.  The sharp interface limit is to minimize the energy
	\beqn\label{eqn:limit_energy}
		\E_{\mu,\sigma}(u):= \frac{1}{2}\int_{\Tn} |\nabla u| + \sigma \int_{\Tn} (u(x)-1)^2\,d\mu(x).
	\eeqn
over functions $u\in BV(\Tn;\{\pm 1\})$ satisfying a mass constraint
	\beqn\label{eqn:mass_const}
		\int_{\Tn} u(x)\,dx = m
	\eeqn
where $m\in(-1,1)$ and $\sigma>0$ are constants and $\Tn$ is the $n$-dimensional flat torus. Here $\int_{\Tn}|\nabla u|$ denotes the \emph{total variation} of the function $u\in BV(\Tn;\{\pm 1\})$ and is defined as 
	\[
		\int_{\Tn}|\nabla u| := \sup \left\{\int_{\Tn} u \dive \varphi\,dx\colon \varphi\in C_0^1(\Tn;\Rn),\,|\varphi(x)|\leq 1 \right\}.
	\]
Hence, the first term in the energy calculates the perimeter of the interface between the phases $1$ and $-1$ whereas the second term penalizes the phase $u=-1$ according to the given probability measure $\mu$. The strength of this penalization is controlled by the parameter $\sigma$. The energy $\E_{\mu,\sigma}$ arises as the singular limit of a sequence of energies that appear in models of self-assembly of nanoparticle-polymer blends and are given by the standard Cahn-Hilliard energy with an inhomogeneous term. This limiting energy can be considered as the extension of the classical isoperimetric problem to an inhomogeneous medium. It also gives a simple model of understanding how penalization of one phase via a probability density affects the geometry of the phase boundary. Indeed, depending on the choice of the measure $\mu$ (and the mass $m$,) the penalization term here can act as an \emph{attraction} or \emph{repulsion} between associated phases. Minimization of energies similar to $\E_{\mu,\sigma}$ where the competition between the interfacial and bulk energetic terms drives a pattern formation has attracted much interest in the past. In particular, energies with a penalization term in this spirit appear in modelling image processing problems such as image segmentation, inpainting and denoising (cf. \cite{ChFoZw2014,GiOs2008,JuKaSh2007,KaMa2013}). This energy also has much in common with the problem of minimizing perimeter in the presence of an obstacle constraint (cf. \cite{BaTa88,BaMa82,BrKi74,Ki73,Gi73}).

\subsection{The Diffuse Interface Model}\label{subsec:model}
To model the nanoparticle-block copolymer configurations with $N$ nanoparticles where each particle is of the form of an $n$-dimensional ball of radius $r>0$ and center $x_i\in\Tn$ for $i=1,\ldots,N$, in \cite{Ginz1,Ginz2} the authors consider a free energy which extends the Ohta-Kawasaki model \cite{OK} with an addition of a penalization term. In a dynamical model, one would expect the nanoparticles to be mobile as they interact with the diblock copolymers. Here we assume that the nanoparticle dynamics is at a significantly slower time scale; hence, we fix their location, and treat them as confining or pinning elements of polymer chains. Neglecting the mobility of nanoparticles is reasonble since we would like to observe the change in the polymer morphology in terms of the minimization of the free energy with respect to the phase parameter. To be precise, in its most general form we consider the free energy
	\beqn\label{eqn:gen_energy}
		\begin{aligned}
			\E_{\eps,\gamma,\eta,m,r,u_p,N}(u;\x)&:= \frac{3\eps}{8}\int_{\Tn}|\nabla u|^2\,dx + \frac{3}{16\eps}\int_{\Tn} (u^2-1)^2\,dx \\
																						&\qquad\qquad+ \frac{\gamma}{2}\int_{\Tn}\!\int_{\Tn} G(x,y)(u(x)-m)(u(y)-m)\,dxdy \\
																						&\qquad\qquad\qquad+ \eta \int_{\Tn} \sum_{i=1}^N V(|x-x_i|)(u-u_p)^2\,dx.
		\end{aligned}
	\eeqn
Here $\Tn$ denotes the $n$-dimensional flat torus, $\x=(x_1,\ldots,x_N)\in\Tn\times\cdots\times\Tn$ denotes the vector consisting of the centers of nanoparticles and $u\in H^1(\Tn)$ is the phase parameter. Moreover, as above $m=\int_{\Tn}u(x)\,dx$, and $\eps>0$, $\gamma>0$, $\eta>0$, $r>0$ and $u_p\in[-1,1]$ are constants. Clearly $m\in(-1,1)$ is related to the volume fraction of the diblock copolymers describing the distribution of different types polymers with $m=0$ corresponding to equal mass distribution between two types of polymers. As it is standard in Cahn-Hilliard-type energies, $\eps>0$ describes the thickness of the transition layer between $u=1$ and $u=-1$.  (The extra factor of $\frac34$ will simplify the form of the eventual Gamma limit, and is inconsequential.)  The parameter $\gamma>0$ is related to the strength of the chemical bond between the polymers subchains in the copolymer macromolecule and controls the long-range interaction between the phases via the Green's function of the flat $n$-torus denoted by $G(x,y)$. The parameter $u_p\in[-1,1]$ determines the ``wetting'' of nanoparticles and their preference towards the polymer phases. For example, $u_p=1$ means that the particles prefer to stick to those subchains of a diblock copolymer macromolecule that are given by the phase $u=1$. Finally $V$ denotes a rapidly decreasing repulsive potential. In \cite{Ginz1,Ginz2}, the authors take $V(|x|)=\exp(-|x|/r_0)$ with $r_0 \ll 1$ so that the repulsion is short-ranged; however, we restrict $V$ to be a smooth and radial function of support $B(0,r)$ so that the interaction between particles is zero. Hence, the last term of the energy \eqref{eqn:gen_energy} simplifies to
	\[
		\eta\,\sum_{i=1}^N \int_{B(x_i,r)} V(|x-x_i|)(u(x)-u_p)^2\,dx.
	\]

\smallskip

\begin{remark}(The potential $V$)
The reason for the inclusion of a potential $V$ as a weight in the penalization term is related to the dynamics of the model. Indeed, in their model the authors consider the evolution of a system of nanoparticle-block copolymer blend as a gradient flow of the free energy given by \eqref{eqn:gen_energy} along with a system of ordinary differential equations for evolution of the centers of nanoparticles. There a nonzero potential $V$ enables the nanoparticles to move around in the domain $\Tn$ whereas its short-range allows the authors to neglect the interaction between nanoparticles.
\end{remark}

Even in the absence of nanoparticles (i.e., when $\eta=0$) the microphase separation of diblock copolymers yields a rather rich and complex picture. There is an extensive literature on the mathematical analysis of phase separation of block copolymers via the Ohta-Kawasaki model and its sharp interface limit leading to a \emph{nonlocal isoperimetric problem}. From mathematical derivation of the model \cite{C2001,CR2005} to analysis on curved manifolds \cite{To2013,ChToTs15}, the energy landscape of \eqref{eqn:gen_energy} with $\eta=0$ and its $\Gamma$-limit as $\eps\to 0$ whether posed on the flat torus (i.e. with periodic boundary conditions), on a general domain with homogeneous Neumann data or on the whole Euclidean space has been rigorously investigated in various parameter regimes of $m$ and $\gamma$ (cf. \cite{AcFuMo13,AlChOt2009,BoCr14,ChPe2010,ChPe2011,GMS2013,GMS2014,KnMu2013,KnMu2014,LO2014,Mu2002,Mu2010,MoSt2014,RW2000,RW2009,RW2014,ST}). However, to our knowledge, mathematical analysis of nanoparticle-block copolymer blends via the energy \eqref{eqn:gen_energy} or its sharp interface version \eqref{eqn:limit_energy} has not been carried out.

\smallskip

\subsection{Choices of parameters}\label{subsec:ch_par}
We concentrate only on the local interactions between the phases, i.e., we choose $\gamma=0$.  This simplification does not affect the passage to the $\Gamma$-limit in the full energy functional, as the nonlocal interaction term is a continuous perturbation of the other terms, but (as explained above) in this paper we restrict our attention to the competition between the interfacial and bulk confinement terms only.
 In the absence of nanoparticles ($\eta=0$) the periodic phase separation and the relation between the periodic Cahn-Hilliard energy and the periodic isoperimetric problem have been investigated in \cite{ChSt2006}. The addition of nanoparticles into the model, of course, poses new challenges. In the model \eqref{eqn:gen_energy}, we take the weight of the penalization term to be compactly supported in a ball, smooth, radial, repulsive and normalized to have mass 1. Specifically, we choose the function $V:\R\to\R$ so that for $\mathcal{V}(x):=V(|x|)$
\medskip
		\begin{itemize}	
				\addtolength{\itemsep}{6pt}
						\item[(A1)] $\supp \mathcal{V} = B(0,r)$ for some $0<r\ll 1$,
						\item[(A2)] $\mathcal{V}\in C^1(\Tn)$,
						\item[(A3)] $V^\pr(|x|)\leq 0$ for all $x\in\Tn$, and
						\item[(A4)] $\int_{\Tn} \mathcal{V}(x)\,dx=1$.
		\end{itemize}
\medskip

Also, we take
	\[
		\eta := \frac{\sigma}{N}
	\]
for some $\sigma>0$. With these choices the energy is $\mathcal{O}(1)$. Moreover, we assume that $u_p=1$, i.e., that the nanoparticles completely prefer the phase $u=1$.  With the above choices, the free energy \eqref{eqn:gen_energy} of a nanoparticle-polymer blend with $N$-many round particles centered at $\x\in\Tn\times\cdots\times\Tn$ with radius $r>0$ is given by
	\beqn\label{eqn:energy}
		\begin{aligned}
			\E_{\eps,\sigma,r,N}(u;\x) &= \frac{3\eps}{8}\int_{\Tn}|\nabla u|^2\,dx + \frac{3}{16\eps}\int_{\Tn} (u^2-1)^2\,dx \\
														 &\qquad\qquad + \frac{\sigma}{N} \sum_{i=1}^N \int_{B(x_i,r)}  V(|x-x_i|)(u-1)^2\,dx.
		\end{aligned}
	\eeqn
We consider the energy $\E_{\eps,\sigma,r,N}$ over functions $u\in H^1(\Tn)$ with $\int_{\Tn}u(x)\,dx=m$ and a set of fixed points $\x\in (\Tn)^N$ as centers of nanoparticles. Note that the constants in front of the first two terms in \eqref{eqn:energy} are chosen so that these two terms together $\Gamma$-converge to the perimeter of the phase $u=1$, namely to the first term of \eqref{eqn:limit_energy}, as $\eps\to 0$ in the $L^1(\Tn)$-topology (cf. \cite{Stern}).

\subsection{Outline}\label{subsec:outline}
Our main result in Section \ref{sec:infin_part} states that for any absolutely continuous probability measure $\mu\in\Prb_{\text{ac}}(\Tn)$ we can find $N_\eps$-many points so that an appropriate extension of the energy $\E_{\eps,\sigma,r_\eps,N_\eps}$ to $L^1(\Tn)$ $\Gamma$-converges to the energy $\E_{\mu,\sigma}$ given by \eqref{eqn:limit_energy} as the number of particles $N_\eps$ tends to infinity and the radius $r_\eps$ of each particle goes to zero as $\eps\arrow0$ (Proposition~\ref{prop:gamma_conv}). To prove this result we exploit the above mentioned $\Gamma$-convergence of the periodic Cahn-Hilliard energy to the isoperimetric energy as the interfacial thickness $\eps$ goes to zero (cf. \cite{ChSt2006,Stern}) while approximating the measure $\mu$ by measures where the density is given as the weight in the penalization term of $\E_{\eps,\sigma,r_\eps,N_\eps}$. A classical consequence of $\Gamma$-convergence is that a sequence of minimizers of the energies $\E_{\eps,\sigma,r_\eps,N_\eps}$ converges to a minimizer of the energy $\E_{\mu,\sigma,r_\eps,N_\eps}$ in $L^1(\Tn)$ as $\eps\to 0$ (Proposition~\ref{prop:min_conv}). Note that the energies $\E_{\mu,\sigma}$ and $\E_{\eps,\sigma}$ admit global minimizers by the direct method of the calculus of variations for any $\mu\in\Prb_{\text{ac}}(\Tn)$, $\sigma>0$ and $\eps>0$. This $\Gamma$-convergence result also extends to the copolymer model with the inclusion of the nonlocal repulsive interaction term (Remark~\ref{rem:nonloc_pert}).	

In Section \ref{sec:properties}, we state regularity and criticality properties of local minimizers of the energy $\E_{\mu,\sigma}$. Indeed, we prove that the phase boundaries of $L^1$-local minimizers of $\E_{\mu,\sigma}$ are regular provided the density of the measure $\mu$ is in $L^\infty$ (Proposition~\ref{prop:reg}). Moreover, under additional smoothness assumptions on the density (namely when the density is $C^1$) we present the first variation of the energy $\E_{\mu,\sigma}$ giving a necessary condition of criticality (Proposition~\ref{prop:firstsecondvar}).

The periodic isoperimetric problem has been the focus of much attention in the past (see e.g. \cite{G-B1996,HPRR2004,HoHuMo1999,MoRo2010} and references therein), and it is well-known that solutions of the isoperimetric problem possess phase boundaries of constant mean curvature. This, of course, may not be the case for the minimizers of the energy $\E_{\mu,\sigma}$. However, exploiting the regularity and criticality results of Section \ref{sec:properties}, in Section \ref{sec:T2} we provide an example in two dimensions to illustrate how the minimizing pattern may be affected by the presence of the  penalization measure $\mu$ (Example~\ref{exmpl:diskpen}).  Chosing $d\mu = {1\over \pi r^2}\chi_{B(0,r)}\, dx$ with appropriately chosen radius $r$ (compared to the mass constraint $m$), we show that the lamellar pattern, which is the global minimizer of the classical isoperimetric problem \cite{HoHuMo1999}, ceases to be a minimizer of energy $\E_{\mu,\sigma}$ for any $\sigma>0$ with this choice of penalization.  Moreover we discuss possible global minimizers of $\E_{\mu,\sigma}$ for this particular $\mu$ and prove that for all sufficiently large $\sigma$, the global minimizer of $\E_{\mu,\sigma}$ is given by a disk inside the support of $\mu$ (Propositions~\ref{prop:large} and~\ref{prop:contract}). These discussions also emphasize how the penalization affects the morphology (and geometry) of the phases (and their boundaries) (see Figure \ref{fig:transitioninsigma}).

We conclude (in Section 5) by pointing out several directions for possible future studies.

\begin{figure}[ht!]
     \begin{center}

        \subfigure[{\tiny $\sigma$=0}]{
            \label{fig:frst}
            \includegraphics[width=0.25\linewidth]{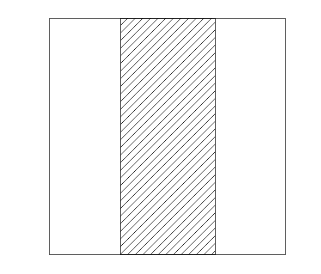}
        }\quad\
        \subfigure[{\tiny $\sigma>0$ small}]{
           \label{fig:scnd}
           \includegraphics[width=0.25\linewidth]{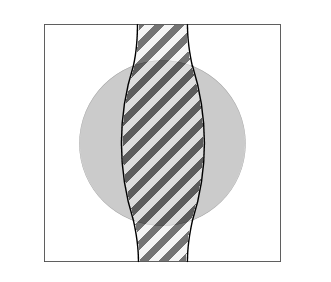}
        }\quad\ 
        \subfigure[{\tiny $\sigma>0$ large}]{
            \label{fig:thrd}
            \includegraphics[width=0.25\linewidth]{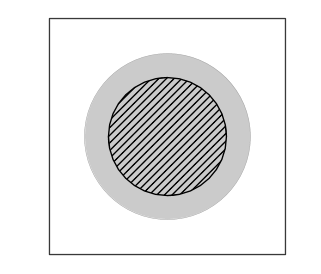}
        }

    \end{center}
    \caption{
        Illustration of transition from a single striped lamellar pattern to a disk as $\sigma$ grows.
        Patterns (a) and (c) are the global minimizers of $\E_{\mu,\sigma}$ for $\sigma=0$ and $\sigma>\sigma_0$, respectively.  Pattern (b) is a likely candidate for a minimizer with $\sigma>0$ but small. The gray disk represents the penalization region $B(0,r)$. 
     }
   \label{fig:transitioninsigma}
\end{figure}

\section{Large Number of Asymptotically Small Particles}\label{sec:infin_part}

In this section we prove that given an absolutely continuous probability measure $\mu\in\Prb_{\text{ac}}(\Tn)$ the energy $\E_{\mu,\sigma}$ appears as the asymptotic limit of the energy $\E_{\eps,\sigma,r,N}$ as the number of particles approaches infinity, and the size of nanoparticles and the thickness of the phase transition go to zero simultaneously. Here the measure $\mu$ gives the probability distribution of infinitely many asymptotically small nanoparticles in the $\eps\to 0$ limit. Indeed, we identify the energy \eqref{eqn:limit_energy} as the $\Gamma$-limit of the energy $\E_{\eps,\sigma,r,N}$ in this asymptotic regime when $r$ and $N$ approach zero and infinity at certain rates, respectively, as functions of $\eps$. To this end, for given $\eps>0$, let $r_\eps=r(\eps)$ and $N_\eps=N(\eps)$ be functions of $\eps$ so that
	\beqn\label{eqn:conv_rates}
		r_\eps\to 0, \qquad N_\eps\to +\infty \quad \text{ and }\quad N_\eps \, r_\eps^n \ll 1
	\eeqn
as $\eps\to 0$.

\smallskip

For fixed $N_\eps$-many points $x_1,\ldots,x_{N_\eps}\in\Tn$ consider the energy $\E_{\eps,\sigma}$ which extends the energy $\E_{\eps,\sigma,r_\eps,N_\eps}$ defined in \eqref{eqn:energy} to $L^1(\Tn)$ as follows
	\beqn\label{eqn:energy_eps}
			\E_{\eps,\sigma}(u):=\begin{cases}
 \frac{3\eps}{8}\int_{\Tn}|\nabla u|^2\,dx + \frac{3}{16\eps}\int_{\Tn} (u^2-1)^2\,dx
 & \text{if } u \in H^1(\Tn)\\
 \qquad + \frac{\sigma}{N_\eps\,r_\eps^n} \sum_{i=1}^{N_\eps} \int_{B(x_i,r_\eps)}  V(|x-x_i|/r_\eps)(u-1)^2\,dx &
   \quad\text{and }   \int_{\Tn} u \,dx = m, \\ \\
+ \infty & \text{otherwise}.
\end{cases}
	\eeqn
Note that here we rescale the weight as $\mathcal{V}_\eps=r_\eps^{-n}V(|x|/r_\eps)$ so that the assumption (A4) is satisfied for all $r_\eps\to 0$ as $\eps\to 0$.

Similarly we will extend the energy $\E_{\mu,\sigma}$, however, with an abuse of notation, we will still denote its extension to $L^1(\Tn)$ by $\E_{\mu,\sigma}$. Namely, for $\mu\in\Prb_{\text{ac}}(\Tn)$ let
	\beqn\label{eqn:limit_energy_ext}
		\E_{\mu,\sigma}(u):=\begin{cases}
 					\frac{1}{2}\int_{\Tn} |\nabla u| + \sigma \int_{\Tn} (u(x)-1)^2\,d\mu(x) &\;\, \text{if } u \in BV(\Tn), u=\pm 1 \text{ a.e.}\\
 &\;\,
   \quad\text{and }   \int_{\Tn} u \,dx = m, \\ \\
+ \infty & \;\, \text{otherwise}.
\end{cases}
	\eeqn
With these definitions we obtain that the family of energies $\E_{\eps,\sigma}$ $\Gamma$-converge to $\E_{\mu,\sigma}$ in the topology of $L^1(\Tn)$ as $\eps\to 0$.
 
\bprop[$\Gamma$-convergence of $\E_{\eps,\sigma}$]\label{prop:gamma_conv}
	Let $\mu\in\Prb_{\text{\emph{ac}}}(\Tn)$ be a probability measure that is absolutely continuous with respect to Lebesgue measure. Let the energies $\E_{\eps,\sigma}$ and $\E_{\mu,\sigma}$ be defined by \eqref{eqn:energy_eps} and \eqref{eqn:limit_energy_ext}. Then there exists $N_\eps$-many points $x_1,\ldots,x_{N_\eps}$ in $\Tn$ such that
	\begin{itemize}
		\item[(i)] \emph{(Lower bound)} for any $u\in L^1(\Tn)$ and for any sequence $\{u_\eps\}_{\eps>0}$ in $L^1(\Tn)$ such that $u_\eps\arrow u$ in $L^1(\Tn)$ as $\eps\arrow 0$, we have
		\[
			\liminf_{\eps\arrow 0} \E_{\eps,\sigma}(u) \geq \E_{\mu,\sigma}(u),
		\]
		and
		\item[(ii)] \emph{(Upper bound)} for any $u\in L^1(\Tn)$ there exists a sequence $\{v_\eps\}_{\eps>0}$ in $L^1(\Tn)$ satisfying
			\[
				v_\eps \arrow u \quad\text{in }L^1(\Tn)
			\]
		and
			\[
				\limsup_{\eps\arrow 0} \E_{\eps,\sigma}(v_\eps) \leq \E_{\mu,\sigma}(u).
			\]
	\end{itemize}
\eprop

\begin{proof} We will prove this proposition in three steps.

\medskip

\noindent \emph{Step 1. (Approximation).} First, we will show that a measure $\mu\in\Prb_{\text{ac}}$ can be approximated in the weak-* topology by a collection of measures distributed on balls of radius $r_\eps>0$ with weight $V$. Similar arguments appear in \cite[Proposition 2.2]{SaSe} and \cite[Lemma 7.5]{JS} for specific rates of convergence of $r_\eps$ and $N_\eps$ depending on the physical model. To begin, for $\mu\in\Prb_{\text{ac}}(\Tn)$ let $\rho\in L^1(\Tn)$ denote its density and suppose that $1/C < \rho(x) < C$ for some $C>0$ and for a.e. $x\in\Tn$. If not, we can consider measures $\mu_k$ defined via densities $\rho_k=(\rho \wedge k)\vee 1/k$ and approximate $\mu_k$ by $\mu_k^\eps$. Taking a diagonal sequence then will yield the result.

Given $\eps>0$, define $k_\eps:=\lfloor \eps^{-1} \rfloor$,  where $\lfloor\cdot\rfloor$ denotes the greatest integer less than its argument, and let $I_{k_\eps}$ denote the index set $\{1,\ldots,2^{nk_\eps}\}$. Define the family of nested cubes $\{Q_i\}_{i\in I_{k_\eps}}$ where 
	\begin{enumerate}[(i)]
		\item $Q_i\subset\Tn$ is a cube of side length $2^{-nk_\eps}$,
		\item $Q_i\cap Q_j=\emptyset$ for $i\neq j$,
		\item $\Tn=\bigcup_{i\in I_{k_\eps}}Q_i$ for all $\eps>0$, and
		\item for $\eps_1<\eps_2$, $I_{k_{\eps_2}}\subset I_{k_{\eps_1}}$ and $\{Q_i\}_{i\in I_{k_{\eps_1}}}\subset\{Q_i\}_{i\in I_{k_{\eps_2}}}$.
	\end{enumerate}
Also, define $d_\eps:=2^{-nk_\eps}$, and along with \eqref{eqn:conv_rates}, suppose that
	\[
		r_\eps \ll N_\eps^{-1/n} \ll d_\eps.
	\]
For each $i \in I_{k_\eps}$, let $N_i^\eps=\lfloor N_\eps\mu(Q_i)\rfloor$. In each $Q_i$ select $N_i^\eps$ points $\{x_{ij}^\eps\}_{j=1}^{N_i^\eps}$ that are almost equally distributed such that the distance between them is of order $d_\eps (N_i^\eps)^{-1/n}$. Noting that $C/d_\eps^n\leq\mu(Q_i)\leq Cd_\eps^n$, this implies that for $j\neq k$ we have
	\beqn\label{eqn:diameter}
		|x_{ij}^\eps-x_{ik}^\eps| > C\,N_\eps^{-1/n} \gg C\,r_\eps.
	\eeqn
Also, note that for each $\eps>0$ the total number of points $N_\eps$ is
	\[
		N_\eps = \sum_{i\in I_{k_\eps}} N_i^\eps.
	\]
	
\smallskip

Let $\mu_i^\eps$ be given such that
	\[
		d\mu_i^\eps(x):=\frac{1}{r_\eps^n}\sum_{j=1}^{N_i^\eps} V(|x-x_{ij}^\eps|/r_\eps)\chi_{B(x_{ij}^\eps,r_\eps)}(x)\,dx
	\]
and define
	\beqn\label{eqn:mu_eps}
		\mu_\eps := \frac{1}{N_\eps}\sum_{i\in I_{k_\eps}} \mu_i^\eps.
	\eeqn
Clearly, $\mu_\eps$ is a probability measure, as $\mu_i^\eps(\Tn)=N_i^\eps$ implies that $\mu_\eps(\Tn)=1$. Moreover, by \eqref{eqn:diameter}, for any $i\in I_{k_\eps}$, $B(x_{ij}^\eps,r_\eps)\subset Q_i$ and $B(x_{ij}^\eps,r_\eps)\cap B(x_{ik}^\eps,r_\eps)=\emptyset$ for $j,k=1,\ldots,N_i^\eps$.

To prove that $\mu_\eps\to\mu$ as $\eps\to 0$ in the weak-* topology it suffices to show that 
	\[
		\limsup_{\eps\to 0}\mu_\eps(\Om)\leq \mu(\Om)
	\]
for any closed set $\Om\subset\Tn$. The weak-* convergence then follows from the Portmanteau Theorem (cf. \cite[Theorem 1.3.4]{VaartWellner}).  

For any closed set $\Om\subset\Tn$, by the assumption (A4), we have that
	\beqn
		\begin{aligned}
		\mu_\eps(\Om) &= \frac{1}{N_\eps} \sum_{i\in I_{k_\eps}}\mu_i^\eps(\Om\cap \ol{Q_i}) \\
					  &= \frac{1}{N_\eps\,r_\eps^n}\sum_{i\in I_{k_\eps}}\sum_{j=1}^{N_i^\eps} \int_{B(x_{ij}^\eps,r_\eps)\cap \Om} V(|x-x_{ij}^\eps|/r_\eps)\,dx \\
					  &\leq \frac{1}{N_\eps\,r_\eps^n}\sum_{\substack{i\in I_{k_\eps} \\ \Om\cap\ol{Q_i}\neq\emptyset}}\sum_{j=1}^{N_i^\eps}  \int_{B(x_{ij}^\eps,r_\eps)} V(|x-x_{ij}^\eps|/r_\eps)\,dx \\
					  &= \frac{1}{N_\eps}\sum_{\substack{i\in I_{k_\eps} \\ \Om\cap\ol{Q_i}\neq\emptyset}} N_i^\eps \\
					  &\leq \sum_{\substack{i\in I_{k_\eps} \\ \Om\cap\ol{Q_i}\neq\emptyset}}\mu(Q_i).
		\end{aligned}
		\nonumber
	\eeqn

Let $\delta>0$ be arbitrary. Define
	\[
		\Om_\eps:= \bigcup_{\substack{i\in I_{k_\eps} \\ \Om\cap\ol{Q_i}\neq\emptyset}} Q_i.
	\]
Since the families $\{Q_i\}_{i\in I_{k_\eps}}$ are nested, we have that $\Om_{\eps_1}\subset\Om_{\eps_2}$ for $\eps_1<\eps_2$. Also, since $\Om$ is closed, $\Om=\bigcap_{\eps>0}\Om_\eps$, and for any given $\delta>0$ there exists $\eps_0>0$ such that for all $\eps<\eps_0$,
	\[
		\sum_{\substack{i\in I_{k_\eps} \\ \Om\cap\ol{Q_i}\neq\emptyset}}\mu(Q_i) \leq \mu(\Om) + \delta.
	\]
Therefore
	\[
		\limsup_{\eps\to 0} \mu_\eps(\Om) \leq \mu(\Om) + \delta,
	\]
and letting $\delta\to 0$ yields the result.

\medskip

\noindent \emph{Step 2. (Lower bound).} After relabelling the points $\{x_{ij}^\eps\}$ found in Step 1, for any $\eps>0$ we obtain a set of $N_\eps$ points $x_i,\ldots,x_{N_\eps}\in\Tn$. For any $u\in L^1(\Tn)$ define the energy $\E_{\eps,\sigma}$ given by \eqref{eqn:energy_eps} using the points $\{x_i\}_{i=1}^{N_\eps}$. Let
	\[
	 P_\eps(u) := \frac{3\eps}{8}\int_{\Tn}|\nabla u|^2\,dx + \frac{3}{16\eps}\int_{\Tn} (u^2-1)^2\,dx,
	\]
and
	\[
		K_\eps(u) := \frac{1}{N_\eps\,r_\eps^n} \sum_{i=1}^{N_\eps} \int_{B(x_i,r_\eps)}  V(|x-x_i|/r_\eps)(u-1)^2\,dx
	\]
so that $\E_{\eps,\sigma}(u)=P_\eps(u)+\sigma K_\eps(u)$.

Let $u\in L^1(\Tn)$ and let $\{u_\eps\}_{\eps>0}\subset L^1(\Tn)$ be a sequence such that $u_\eps\arrow u$ in $L^1(\Tn)$ as $\eps\to 0$. We can assume, without loss of generality, that $u=\pm 1$ a.e. Otherwise
	\[
		\liminf_{\eps\to 0} \E_{\eps,\sigma}(u_\eps) \geq \frac{3}{16\eps}\int_{\Tn} (u_\eps^2-1)^2\,dx=+\infty,
	\]
and the result of Part (i) follows trivially. Similarly, if $\int_{\Tn}u\,dx\neq m$, then for small $\eps>0$, $\int_{\Tn} u_\eps\,dx\neq m$ and $\liminf_{\eps\to 0} \E_{\eps,\sigma}(u_\eps)=+\infty$. Therefore it suffices to consider only functions $u\in L^1(\Tn)$ satisfying $u=\pm 1$ a.e. and $\int_{\Tn} u\,dx =m$.

Moreover, assume that $-1\leq u_\eps \leq 1$ a.e. If not, we can consider the truncated functions
	\beqn
			\ol{u}_\eps = \begin{cases} -1 &\mbox{ on }\quad \{x\colon u_\eps(x)<-1\} \\
																	u_\eps &\mbox{ on }\quad \{x\colon -1\leq u_\eps(x) \leq 1\} \\
																	1 &\mbox{ on }\quad \{x\colon u_\eps(x)>1\}.\end{cases}
		\nonumber
	\eeqn
Clearly $\ol{u}_\eps \arrow u$ in $L^1(\Tn)$ as $\eps\arrow 0$. Also, $P_\eps(u_\eps)\geq P_\eps(\ol{u}_\eps)$ and $K_\eps(u_\eps)\geq K_\eps(\ol{u}_\eps)$. Hence, we can use the truncated functions $\ol{u}_\eps$ instead of $u_\eps$.

\smallskip

In \cite[Section B]{Stern}, the author shows that $\liminf_{\eps\arrow 0}P_\eps(u_\eps) \geq P(u)$. Now we will show that a similar lower semi-continuity property also holds for the penalization term $K_\eps$. Note that, by \eqref{eqn:mu_eps} we can write
	\[
		K_\eps(u) = \int_{\Tn} (u_\eps-1)^2\,d\mu_\eps(x).
	\]
Also note that since $u_\eps\arrow u$ in $L^1(\Tn)$ and $-1\leq u_\eps \leq 1$ we have that $u_\eps\arrow u$ in any $L^p(\Tn)$ with $p\geq1$. In particular, $(u_\eps-1)^2\arrow(u-1)^2$ a.e. as $\eps\arrow 0$. Then by the Egoroff's Theorem, for any arbitrary $\delta>0$ there exists a compact set $G_\delta\subset\Tn$ such that $(u_\eps-1)^2\arrow(u-1)^2$ uniformly on $G_\delta$ and $|\Tn\setminus G_\delta|<\delta$. Moreover, there exists $\eps_0>0$ such that for any $\eps<\eps_0$
	\[
		\Big|(u_\eps-1)^2-(u-1)^2\Big| < \frac{\delta}{|G_\delta|}
	\]
for all $x\in G_\delta$.

On the other hand, since $u=\pm 1$ a.e., $(u-1)^2=4\chi_{A^c}$ where $A=\{x\in\Tn\colon u(x)=1\}$ and $A^c$ denotes the complement of $A$, namely, $\Tn\setminus A$. Since the set $A$ has finite perimeter it holds that $|\pt A|=0$, and since $\mu$ is absolutely continuous with respect to the Lebesgue measure we get that $\mu(\pt A)=0$. Thus, $A$ (and $A^c$), is a continuity set of the measure $\mu$. Since $\mu_\eps\to\mu$ in the weak-* topology, this implies via the Portmanteau Theorem (again, cf. \cite[Theorem 1.3.4]{VaartWellner}) that
	\[
		\lim_{\eps\to 0} \mu_\eps(A^c) = \mu(A^c).
	\]  
Combining these, we get that
	\beqn
		\begin{aligned}
			\liminf_{\eps\arrow 0} \int_{\Tn} (u_\eps-1)^2\,d\mu_\eps(x) &\geq \liminf_{\eps\arrow 0}\int_{G_\delta} (u_\eps-1)^2\,d\mu_\eps(x) \\
&= \liminf_{\eps\to 0} \int_{G_\delta}\Big((u_\eps-1)^2-(u-1)^2+(u-1)^2\Big)\,d\mu_\eps(x) \\
&\geq \liminf_{\eps\to 0} \int_{G_\delta} (u-1)^2\,d\mu_\eps(x) - \delta \\
&=\int_{G_\delta} (u-1)^2\,d\mu(x) - \delta \\
&\geq \int_{\Tn} (u-1)^2\,d\mu(x) - C\delta
		\end{aligned}
		\nonumber
	\eeqn
for some constant $C>0$ independent of $\eps>0$. Therefore letting $\delta\arrow 0$ yields
	\[
		\liminf_{\eps\arrow 0} K_\eps(u_\eps) \geq K(u);
	\]
hence, Part (i) follows.

\medskip

\noindent \emph{Step 3. (Upper bound).} Let $u\in L^1(\Tn)$ and let the points $\{x_i\}_{i=1}^{N_\eps}\subset\Tn$ be given as above. Assume that $u\in BV(\Tn)$, $u=\pm 1$ a.e. and $\int_{\Tn} u(x)\,dx=m$. Otherwise, $\E_{\mu,\sigma}(u)=\infty$ and by choosing $v_\eps=u$ for all $\eps>0$ the result of Part (ii) follows. Let the set $A\subset\Tn$ be such that
	\[
	 u(x) = \begin{cases} 1 &\mbox{ if }\,x\in A \\
	 										 -1 &\mbox{ if }\,x\in A^c.
	 				\end{cases}
	\]
Let $\Gamma=\pt A \cap \pt A^c$ and assume that $\Gamma\in C^2$. If not, one can approximate $A$ by a sequence of open sets $\{A_k\}_{k\in\mathbb{N}}$ as in \cite[Lemma 1]{Stern} satisfying the condition that $\pt A_k$ is of class $C^2$. Then one can use the sets $\pt A_k\cap\pt A_k^c$ instead of $\Gamma$ to prove Part (ii) and then pass to a limit using a diagonal sequence.

Define the signed distance function $d_\Gamma:\Tn\arrow \R$ by
	  \[
	  	d_\Gamma (x) := \begin{cases} \dist(x,\Gamma) &\mbox{ if }\,x\in A^c \\
	 										 -\dist(x,\Gamma) &\mbox{ if }\,x\in A,
	 				            \end{cases}
	  \]
and the sequence of functions $g_\eps:\R\arrow\R$ by
		\[
			g_\eps(t):= \begin{cases}
											-1                                                        &\mbox{ if } t>2\sqrt{\eps} \\
											\frac{-1-z(1/\sqrt{\eps})}{\sqrt{\eps}}(t-2\sqrt{\eps})-1 &\mbox{ if } \sqrt{\eps}\leq t \leq 2\sqrt{\eps} \\
											z(t/\eps) 																                &\mbox{ if }|t|\leq \sqrt{\eps} \\
											\frac{z(-1/\sqrt{\eps})-1}{\sqrt{\eps}}(t+2\sqrt{\eps})+1	&\mbox{ if } -2\sqrt{\eps}\leq t \leq -\sqrt{\eps} \\
											1																					&\mbox{ if } s<-2\sqrt{\eps}
									\end{cases}
		\]
where the function $z(t)$ solves the ordinary differential equation
		\[
			\frac{dz}{dt}=\frac{\sqrt{3}}{4}(z^2-1) \quad\text{ subject to }z(0)=0.
		\]
In \cite{Stern} the author shows that the function defined by
		\beqn\label{eqn:v_eps}
			v_\eps(x) := g_\eps(d_\Gamma(x)) + \eta_\eps,
		\eeqn
where $\eta_\eps$ is an additive constant that is of order $\mathcal{O}(\eps)$, is in $H^1(\Tn)$, and satisfies the mass constraint $\int_{\Tn} v_\eps(x)\,dx=m$ for all $\eps>0$. Moreover,
	\[
		v_\eps\arrow u\text{ in } L^1(\Tn)\text{ as }\eps\to 0\quad\text{ and }\quad\limsup_{\eps\arrow 0} P_\eps(v_\eps) \leq P(u).
	\]
We will use the family of functions $v_\eps$'s to prove that a similar $\limsup$ inequality also holds true for the penalization term $K_\eps$. To this end let
	\[
		A_+^\eps:=\{x\in\Tn\colon g_\eps(d_\Gamma(x))=1\}\quad\text{ and }\quad A_-^\eps:=\{x\in\Tn\colon g_\eps(d_\Gamma(x))=-1\},
	\]
and let $\Gamma_\eps$ denote the transition layer
	\[
	  \Gamma_\eps:= \{x\in\Tn\colon -1<g_\eps(d_\Gamma(x))<1\}
	\]
Define
	\[
		\Gamma_+^\eps:=\Gamma_\eps \cap A\quad\text{ and }\quad\Gamma_-^\eps:=\Gamma_\eps \cap A^c.
	\]
Then, for any $\eps>0$, $\Tn = A_+^\eps \cup \Gamma_+^\eps \cup \Gamma_-^\eps \cup A_-^\eps$, $A=\Gamma_+^\eps \cup A_+^\eps$ and $A^c = \Gamma_-^\eps \cup A_-^\eps$.

Using the fact that $\eta_\eps=\mathcal{O}(\eps)$ and $g_\eps(d_\Gamma(\cdot))\in L^\infty(\Tn)$, for the functions $v_\eps$ given by \eqref{eqn:v_eps} we get that
	\beqn
		\begin{aligned}
			K_\eps(v_\eps)&=\int_{\Tn} (v_\eps-1)^2\,d\mu_\eps(x) \\
								    &= \int_{\Gamma_+^\eps} \Big(g_\eps(d_\Gamma(x))-1\Big)^2\,d\mu_\eps(x) + \int_{\Gamma_-^\eps} \Big(g_\eps(d_\Gamma(x))-1\Big)^2\,d\mu_\eps(x) + 4\int_{A_-^\eps}d\mu_\eps(x) + \mathcal{O}(\eps) \\
								    &= \int_{\Gamma_+^\eps} \Big(g_\eps(d_\Gamma(x))-1\Big)^2\,d\mu_\eps(x) + \int_{\Gamma_-^\eps} \Big(\big(g_\eps(d_\Gamma(x))-1\big)^2-4\Big)\,d\mu_\eps(x) + 4\int_{A^c}d\mu_\eps(x) + \mathcal{O}(\eps)\\
								    &\leq 4\mu_\eps(\ol{\Gamma_+^\eps}) + 4\int_{A^c}d\mu_\eps(x) + \mathcal{O}(\eps).
		\end{aligned}
		\nonumber
	\eeqn
Let $\delta>\eps>0$ be fixed but arbitrary. Then $\Gamma_{\pm}^\eps \subset \Gamma_{\pm}^\delta$, and for $\Gamma_\delta=\Gamma_+^\delta \cup \Gamma_-^\delta$ we have
	\[
		K_\eps(v_\eps) \leq 4\mu_\eps(\ol{\Gamma_\delta}) + 4\int_{A^c}d\mu_\eps(x) + \mathcal{O}(\eps).
	\]
Again, by the Portmanteau Theorem the weak-* convergence of $\mu_\eps$ to $\mu$ is equivalent to the fact that $\limsup_{\eps\to 0}\mu_\eps(\ol{\Gamma_\delta})\leq \mu(\ol{\Gamma_\delta})$ as $\ol{\Gamma_\delta}$ is closed in $\Tn$. Moreover, since $|\pt A|=0$ and $A$ is a continuity set for the measure $\mu$ by its absolute continuity with respect to the Lebesgue measure, as in Step 2, we have that $\lim_{\eps\to 0} \mu_\eps(A^c) = \mu(A^c)$. Therefore,
	\[
		\limsup_{\eps \arrow 0} K_\eps(v_\eps) \leq 4\mu(\ol{\Gamma_\delta}) + 4\int_{A^c}d\mu(x) \leq C\delta + K(u)
	\]
for some constant $C>0$ independent of $\eps>0$. Letting $\delta\arrow 0$ and combining this with $\limsup_{\eps\arrow 0}P_\eps(v_\eps)\leq P(u)$ we obtain the result of Part (ii).
\end{proof}

\medskip

\begin{remark}(Nonlocal perturbations and the diblock copolymer model)\label{rem:nonloc_pert}
Define the nonlocal perturbations of the functionals $\E_{\eps,\sigma}$ and $\E_{\mu,\sigma}$ respectively by
	\beqn\label{eqn:energy_gamma}
			\E_{\eps,\sigma,\gamma}(u):=\begin{cases}
 \frac{3\eps}{8}\int_{\Tn}|\nabla u|^2\,dx + \frac{3}{16\eps}\int_{\Tn} (u^2-1)^2\,dx
 & \text{if } u \in H^1(\Tn)\\
 \qquad  + \frac{\sigma}{N_\eps\,r_\eps^n} \sum_{i=1}^{N_\eps} \int_{B(x_i,r_\eps)}  V(|x-x_i|/r_\eps)(u-1)^2\,dx  & \quad\text{and }   \int_{\Tn} u \,dx = m,  \\
  \qquad\qquad + \gamma \int_{\Tn}\!\int_{\Tn} G(x,y)\big(u(x)-m\big)\big(u(y)-m\big)\,dxdy & \\ \\
+ \infty &  \text{otherwise},
\end{cases}
	\eeqn
and
	\beqn\label{eqn:limit_energy_gamma}
		\E_{\mu,\sigma,\gamma}(u):=\begin{cases}
 					\frac{1}{2}\int_{\Tn} |\nabla u| + \sigma \int_{\Tn} (u(x)-1)^2\,d\mu(x) & \text{if } u \in BV(\Tn), u=\pm 1 \text{ a.e.}\\
 \qquad +\gamma \int_{\Tn}\!\int_{\Tn} G(x,y)\big(u(x)-m\big)\big(u(y)-m\big)\,dxdy & \quad\text{and }   \int_{\Tn} u \,dx = m, \\ \\
+ \infty &  \text{otherwise}
\end{cases}
	\eeqn
over functions $u\in L^1(\Tn)$. Then a standard conclusion of $\Gamma$-convergence is that the convergence is stable under continuous perturbations, i.e., $\E_{\eps,\sigma,\gamma}$ $\Gamma$-converges to $\E_{\mu,\sigma,\gamma}$ as $\eps\to 0$ in the $L^1(\Tn)$-topology (cf. \cite[Proposition 6.21]{DalMaso}).
\end{remark}

\bigskip

Another classical consequence of $\Gamma$-convergence is that the limit of a convergent sequence of energy minimizers minimizes the limiting energy. The proof of the following proposition is quite standard and can be adapted, for example, from the proof of \cite[Theorem 1]{Stern}.
 
\bprop[Limit of a sequence of minimizers]\label{prop:min_conv}
	Let $\mu\in\Prb_{\text{\emph{ac}}}(\Tn)$ and let $x_1,\ldots,x_{N_\eps}$ be such that the sequence of measures $\{\mu_\eps\}_{\eps>0}\subset\Prb_{\text{\emph{ac}}}(\Tn)$ defined via the densities
		\beqn\label{eqn:mu_eps_dens}
			\frac{1}{N_\eps\,r_\eps^n}\sum_{i=1}^{N_\eps} V(|x-x_i|/r_\eps)\chi_{B(x_i,r_\eps)}(x)
		\eeqn
converges to $\mu$ in the weak-* topology of $\Prb(\Tn)$ as $\eps\to 0$ and $N_\eps\to+\infty$. Suppose $u_\eps\arrow u$ in $L^1(\Tn)$ as $\eps\to 0$ where, for any $\eps>0$, $u_\eps$ minimizes the energy $\E_{\eps,\sigma}$ for any $\sigma>0$. Then $u$ minimizes the energy $\E_{\mu,\sigma}$ over $BV(\Tn)$ with $u=\pm 1$ a.e. and $\int_{\Tn} u\,dx=m$. 
\eprop

\begin{remark}
Note that given any measure $\mu\in\Prb_{\text{ac}}$ we can find $N_\eps$-many points as in the Step 1 of the proof of Proposition \ref{prop:gamma_conv} so that the probability measures $\mu_\eps$ defined via the densities \eqref{eqn:mu_eps_dens} converge to $\mu$ in the weak-* topology.	
\end{remark}

\medskip

\begin{remark}[Compactness of a sequence of minimizers]\label{rem:comp_min}
Using the polynomial growth of the double-well potential $(u^2-1)^2$ and the compactness of $BV$-functions in $L^1$ (cf. \cite[Theorem 1.19]{Giusti}) we can easily conclude that if $\{u_\eps\}_{\eps>0}$ is a sequence of minimizers of the energies $\E_{\eps,\sigma}$ then there exists a subsequence $\{u_{\eps_j}\}_{j\in\mathbb{N}}$ such that $u_{\eps_k}\to u_0$ in $L^1(\Tn)$ as $\eps_j\to 0$ for some function $u_0\in L^1(\Tn)$ (cf. \cite[Proposition 3]{Stern}).
\end{remark}

\section{Properties of Local Minimizers of $\E_{\mu,\sigma}$}\label{sec:properties}

Independent of its connection to nanoparticle-polymer models, the energy $\E_{\mu,\sigma}$ also piques one's interest as a rather simple extension of the classical periodic isoperimetric problem where one tries to minimize the perimeter of a set of fixed mass with respect to a penalization term determined by a fixed probability measure. Indeed, as a model for pattern formation, the minimization of \eqref{eqn:limit_energy} sets up a basic competition between short-range effects of the perimeter term and possibly long-range penalization via the choice of the measure $\mu$. The interplay between these competing terms appears in properties of local minimizers such as regularity, criticality and stability.

For $\mu\in\Prb_{\text{ac}}(\Tn)$ let us denote its density by $\rho\in L^1(\Tn)$ for the remainder of this section, i.e., let 
	\[
	d\mu(x)=\rho(x)\,dx.
	\]
With a slight abuse of notation we will also denote by $\E_{\mu,\sigma}$ the energy defined on sets of finite perimeter. Namely, we consider the problem
	\beqn\label{eqn:set_energy}
		\text{locally minimize}\quad \E_{\mu,\sigma}(A) := \int_{\Tn} |\nabla \chi_A| + 4\sigma \int_{A^c} \rho(x)\,dx
	\eeqn 
over sets of finite perimeter $A\subset\Tn$ such that $|A|=(1+m)/2$ for any given $m\in(-1,1)$, $\sigma>0$ and $\rho\in L^1(\Tn)$. Note that this formulation is equivalent to minimizing \eqref{eqn:limit_energy} subject to the constraint $\int_{\Tn}u(x)\,dx=m$. Let us also note that a set of finite perimeter $\Om$ is an $L^1$-local minimizer of $\E_{\mu,\sigma}$ if
	\beqn\label{eqn:localmin}
		\E_{\mu,\sigma}(\Om) \leq \E_{\mu,\sigma}(A) \quad\text{provided}\quad \int_{\Tn} |\chi_\Om-\chi_A|\,dx < \delta
	\eeqn
for some $\delta>0$.

As we noted in the introduction the minimization problem \eqref{eqn:set_energy} is in the spirit reminiscent of finding minimal boundaries with respect to an obstacle set. These obstacle problems tackle the following minimization problem:
	\beqn\label{eqn:obst_prob}
		\text{minimize}\quad \int_{\Tn} |\nabla \chi_A|
	\eeqn
over sets of finite perimeter $A\subset\Tn$ such that $L\subset A$ where the obstacle $L$ is a given fixed set of finite perimeter. Note that here the admissible sets are not mass constrained as they are in our problem. We believe that such an isoperimetric obstacle problem is equivalent to \eqref{eqn:set_energy} in the limit $\sigma\to\infty$ when $|A|=|L|$; however, when minimizing $\E_{\mu,\sigma}(A)$ we do not explicitly restrict the admissible patterns $A$ to those which must contain $\supp\rho$.

We note that the regularity of phase boundaries of the obstacle problem \eqref{eqn:obst_prob} were established in \cite[Section 3]{Tam} depending on the boundary regularity of the obstacle set $L$. For our problem \eqref{eqn:set_energy}, on the other hand, we prove that the regularity of the phase boundaries are determined by controlling the $L^\infty$-bound of the density $\rho$ rather than the smoothness of the boundary of its support. The issue here is to control the “excess-like” quantity \eqref{eqn:excess} that measures how far a set is from minimizing perimeter in a ball in terms of the radius of that ball. Indeed, we show that if $\mu$ has \emph{bounded density} $\rho$ then for any $\sigma>0$ the penalization term can locally be controlled by the perimeter term and we can conclude by the well-established regularity theory for the isoperimetric problem that the phase boundary of a local minimizer of $\E_{\mu,\sigma}(A)$ is of class $C^{1,\alpha}$. The essential elements of the regularity result are already contained in the works of others in the similar context; however, we are unaware of a particular result that applies to our setting specifically. Hence, for completeness, we present here a proof for the regularity of phase boundaries.

	\bprop[Regularity of Phase Boundaries]\label{prop:reg}
If $\rho\in L^\infty(\Tn)$ and $\Omega\subset\Tn$ is an $L^1$-local minimizer of \eqref{eqn:set_energy}, then $\partial^*\Omega$ is of class $C^{1,\alpha}$ for some $\alpha\in (0,1)$ and $\mathcal{H}^s(\pt\Omega\setminus\pt^*\Omega)=0$ for every $s>n-8$ where $\pt^* \Omega$ denotes the reduced boundary of $\Omega$ and $\mathcal{H}^s$ denotes the $s$-dimensional Hausdorff measure.
	\eprop
	
The proof of the regularity of local minimizers of $\E_{\mu,\sigma}(A)$ rely on the following technical lemma proof of which can be found in \cite[Lemma 2.1]{G}.

	\blemma\label{lem:reg}
		Let $L$ be a Borel set, and let $D$ be an open domain such that $\int_D |\nabla\chi_L|>0$. Then there exists positive constants $k_0$ and $l_0$ depending only on $D$ and $D\cap L$ such that for all $k$ with $|k|<k_0$ there exists a set $F$ such that $F=L$ outside of $D$ and
			\begin{gather*}
				|F|=|L| + k \\
				\int_D |\nabla\chi_F| \leq \int_D |\nabla\chi_L| + l_0|k|\\
				\int_D |\chi_F-\chi_L|\,dx \leq l_0 |k| \int_D |\nabla\chi_L|.
			\end{gather*}
	\elemma

Now we prove the regularity result proceeding as in \cite[Proposition 2.1]{ST}.	
	
	\begin{proof}[Proof of Proposition \ref{prop:reg}.]
Let $\Omega$ be an $L^1$-local minimizer of \eqref{eqn:set_energy} and $x_0\in\pt\Om$ be arbitrary. Let $D\subset\subset\Tn$ be such that $x_0\not\in\ol{D}$ and $\int_D |\nabla\chi_\Om|>0$. For $L=\Om$ in Lemma \ref{lem:reg} there exist two constants $k_0$ and $l_0$ that depend only on $D$ and $D\cap\Om$. Using these constants fix $R>0$ such that
	\beqn\label{eqn:radius}
  \omega_n R^n<k_0,\quad \left(1+l_0\int_D |\nabla\chi_{\Om}|\right)\omega_n R^n<\delta\quad\mbox{and}\quad \clB\cap\ol{D}=\emptyset,
  \eeqn
where $\omega_n$ is the measure of the unit $n$-ball and $\delta$ is as in \eqref{eqn:localmin}.

Let $\widetilde{\Om}$ minimize the perimeter in $B_R(x_0)$ subject to the boundary values of $\Om$, i.e.,
	\[
		\int_{B_R(x_0)}|\nabla\chi_{\widetilde{\Om}}| \leq \int_{B_R(x_0)}|\nabla\chi_A|	
	\]
for all $A\subset\Tn$ such that $A\setminus B_R(x_0)=\Om\setminus B_R(x_0)$.

Since $\widetilde{\Om}\cap \ol{D}=\Om\cap\ol{D}$ the result of Lemma \ref{lem:reg} holds true with the same constants $k_0$ and $l_0$ if we replace $\Om$ by $\widetilde{\Om}$. Hence, for $k:=|\Om|-|\widetilde{\Om}|\leq \omega_n R^n < k_0$ by the choice of $R>0$, there exists a set $G$ such that $G=\widetilde{\Om}$ outside $D$ and
	\begin{align}
  |G| &=|\Omega|=m, \label{vol}\\
  \int_D |\nabla\chi_G| &\leq \int_D |\nabla \chi_{\widetilde{\Om}}| + C\,R^n,\label{per}\\
  \int_{\Tn} |\chi_G-\chi_\Omega|\,dx&\leq C_0\,R^n<\delta,\label{L1},
 \end{align} 
where \eqref{L1} follows from \eqref{eqn:radius} with $C_0:=\left(1+l_0\int_D |\nabla\chi_{\Omega}|\right)\omega_n$.

By \eqref{vol} and \eqref{L1}, the set $G$ is an admissible competitor for the energy $\E_{\mu,\sigma}(A)$; hence,
	\[
		\int_{\Tn} |\nabla\chi_\Om| + 4 \sigma \int_{\Om^c} \rho(x)\,dx \leq \int_{\Tn} |\nabla \chi_G| + 4\sigma \int_{G^c} \rho(x)\,dx.
	\]
Noting that $\widetilde{\Om}\setminus B_R(x_0)=\Om\setminus B_R(x_0)$ and $G\setminus D = \widetilde{\Om}\setminus D$, and using \eqref{per} we get that
	\beqn
 		\begin{aligned}
  		\int_{\Tn\setminus(D\cup\B)}|\nabla \chi_\Omega| &+ \int_{D}|\nabla \chi_{\widetilde{\Om}}| + \int_{\B}|\nabla \chi_\Omega| + 4\sigma \int_{\Om^c} \rho(x)\,dx \\
       &\leq \int_{\Tn\setminus(D\cup\B)}|\nabla \chi_G| + \int_{D}|\nabla \chi_G| \\
       &\qquad\qquad+ \int_{\B}|\nabla \chi_G| + 4\sigma \int_{G^c} \rho(x)\,dx \\
       &= \int_{\Tn\setminus(D\cup\B)}|\nabla \chi_\Omega| + \int_{D}|\nabla \chi_G|\\
       &\qquad\qquad+\int_{\B}|\nabla \chi_{\widetilde{\Om}}| + 4\sigma \int_{G^c} \rho(x)\,dx \\
       &\leq \int_{\Tn\setminus(D\cup\B)}|\nabla \chi_\Omega|+\int_{D} |\nabla \chi_{\widetilde{\Om}}| \\
       &\qquad\qquad+ \int_{\B}|\nabla \chi_{\widetilde{\Om}}| + 4 \sigma \int_{G^c} \rho(x)\,dx +
       C\,R^n
 		\end{aligned}
 		\nonumber
 	\eeqn
for some constant $C>0$. Hence, we have
	\beqn\label{eqn:firstest}
 		\int_{\B} |\nabla \chi_\Omega| -\int_{\B} |\nabla\chi_{\widetilde{\Om}}|\leq 4\sigma \left(\int_{G^c}\rho(x)\,dx-\int_{\Om^c}\rho(x)\,dx\right)+ C\,R^n.
	\eeqn
	
On the other hand, using \eqref{L1} and the fact that $\rho\in L^\infty(\Tn)$ we obtain
	\beqn\label{eqn:secondest}
		\begin{aligned}
			\int_{G^c}\rho(x)\,dx-\int_{\Om^c}\rho(x)\,dx &= \int_{\Om}\rho(x)\,dx-\int_{G}\rho(x)\,dx \\
																						  &= \int_{\Tn} \rho(x)(\chi_\Om-\chi_G)(x)\,dx \\
																						  &\leq \|\rho\|_{L^\infty(\Tn)} \|\chi_\Om-\chi_G\|_{L^1(\Tn)} \leq C\,R^n.
		\end{aligned}
	\eeqn
Combining \eqref{eqn:firstest} and \eqref{eqn:secondest}, we get that
  \beqn\label{eqn:excess}
 		\int_{\B} |\nabla \chi_\Omega| -\int_{\B} |\nabla \chi_{\widetilde{\Om}}|\leq C\,R^n.
	\eeqn
Property \eqref{eqn:excess} states that the boundary of the set $\Omega$ is \emph{almost area-minimizing} in any ball. With this property, the classical regularity results of \cite{M,Tam} apply, and we can conclude that $\pt^*\Omega$ is of class $C^{1,\alpha}$, with $\mathcal{H}^s(\pt\Omega\setminus\pt^*\Omega)=0$ for every $s>n-8$.
\end{proof}

\bigskip

With the regularity of phase boundaries at hand, under further smoothness assumptions on the density $\rho$ we have the following necessary condition of local minimality.

\bprop[Criticality Condition]\label{prop:firstsecondvar}
	If $\rho\in C^1(\Tn)$ and $u$ is an $L^1$-local minimizer of the energy $\E_{\mu,\sigma}$, then
		\beqn\label{eqn:firstvar}
			(n-1)H(x) - 4\sigma\rho(x) = \lambda \quad\text{for all }x\in\pt A
		\eeqn
for some constant $\lambda$ where $H:\pt A\arrow\R$ denotes the mean curvature of $\pt A$ and $A=\{x\colon u(x)=1\}$ as before.
\eprop

\begin{proof}
Suppose $\rho\in C^1(\Tn)$ and let $u$ be an $L^1$-local minimizer of $\E_{\mu,\sigma}$. Let $A=\{x\colon u(x)=1\}$, and let $\zeta\in C^\infty(\pt A)$ such that $\int_{\pt A} \zeta(x)\,d\Hn(x)=0$. 

To compute the first variation of the energy $\E_{\mu,\sigma}$, we view it as a set functional given by \eqref{eqn:set_energy} and proceed as in \cite{CS2,SZ1}.

Let
	\[
		X(x)=\zeta(x)\nu(x)\quad\text{on}\quad\pt A
	\]
where $\nu$ denotes the outer unit normal to $\pt A$. Then, clearly, 	
	\beqn\label{eqn:Xmass}
		\int_{\pt A} X\cdot\nu\,d\Hn(x)=0.
	\eeqn
Let $\Psi:\Tn\times(-\tau,\tau)\arrow\Tn$ solve
	\beqn\label{eqn:Psiflow}
		\begin{cases}\frac{\pt\Psi}{\pt t}= X(\Psi) \\	\Psi(x,0)=x,\end{cases}
	\eeqn
for some $\tau>0$. Define
	\[
		A_t := \Psi(A,t).
	\]
	
Invoking Proposition \ref{prop:reg}, we easily see that for the family of sets $\{A_t\}_{t\in(-\tau,\tau)}$ for some $\tau>0$ we have that
	\begin{gather}
		\pt A_t\text{ is of class }C^1,\text{ and}\nonumber\\
		\chi_{A_t}\arrow\chi_A\quad\text{as }t\arrow 0\text{ in }L^1(\Tn).\nonumber
	\end{gather}

Moreover, 
	\beqn\label{eqn:Psiexpan}
		D\Psi(\cdot,t) = I + t\,\nabla X + \frac{1}{2}t^2\,\nabla Z + o(t^2),
	\eeqn
where $Z:=\pt^2\Psi/\pt t^2 |_{t=0}$ is given with $i$-th component $Z^{(i)}=\sum_{j=1}^n X_{x_j}^{(i)} X^{(j)}$, and
	\beqn\label{eqn:Jacobatzero}
		\frac{\pt}{\pt t}\Big|_{t=0} J\Psi = \text{trace}\nabla X = \dive X,
	\eeqn
where $J\Psi$ denotes the Jacobian of $\Psi$. Hence, using \eqref{eqn:Xmass} and the Divergence Theorem, we get that
	\[
		\frac{d}{dt}\Big|_{t=0}|A_t| = \frac{d}{dt}\Big|_{t=0} \int_A J\Psi\,dx = 0,
	\]
i.e., the family of sets $\{A_t\}_{t\in(-\tau,\tau)}$ preserves the volume of $A$ to first order. Therefore this family of sets is an admissible class of perturbations of $A$ to compute the first variation of $\E_{\mu,\sigma}(A)$.

Define the functions $U(x,t)$ by
	\beqn\label{eqn:Ut}
		U(x,t)=\begin{cases} 1 &\mbox{if } x \in A_t, \\ 
-1 & \mbox{if } x \in A_t^c, \end{cases}
	\eeqn
and note that a function $u\in BV(\Tn,\{\pm1\})$ is said to be a \emph{critical point} of the energy $\E_{\mu,\sigma}$ if $d/dt|_{t=0} \E_{\mu,\sigma}(U(\cdot,t))=0$ for every $U(x,t)$ defined via the admissible family $\{A_t\}$.

\medskip

Consider the energy
	\beqn
		\begin{aligned}
		\E_{\mu,\sigma}(U(\cdot,t))&=\frac{1}{2}\int_{\Tn} |\nabla U(\cdot,t)| + \sigma \int_{\Tn} (U(x,t)-1)^2\rho(x)\,dx \\
									&=: P(t) + \sigma K(t).
		\end{aligned}
		\nonumber
	\eeqn
In \cite{CS2}, the authors show that
	\beqn\label{eqn:first_var_per}
		P^\pr(0) = (n-1)\int_{\pt A} H(x)\zeta(x)\,d\Hn(x)
	\eeqn
where $H$ denotes the mean curvature of $\pt A$. Now we are going to compute $K^\pr(0)$. Note that, by \eqref{eqn:Psiflow},
	\beqn
		\begin{aligned}
		K(t) &= 4\int_{A_t^c}\rho(x)\,dx \\
				 &= 4\int_{\Tn}\rho(x)\,dx-4\int_{A_t}\rho(x)\,dx \\
				 & = 4\int_{\Tn}\rho(x)\,dx-4\int_{A}\rho(\Psi(x,t))J\Psi(x,t)\,dx
		\end{aligned}
		\nonumber
	\eeqn
Therefore, since $\rho\in C^1(\Tn)$,
	\beqn\label{eqn:Kprime}
		K^\pr(t)=-4\int_A \nabla\rho(\Psi(x,t))\frac{\pt}{\pt t}\Psi(x,t)J\Psi(x,t) + f(\Psi(x,t))\frac{\pt}{\pt t}(J\Psi(x,t))\,dx.
	\eeqn
Hence, by \eqref{eqn:Psiflow}, \eqref{eqn:Psiexpan} and \eqref{eqn:Jacobatzero}, using the Divergence Theorem we get that
	\beqn
		\begin{aligned}
		K^\pr(0) &= -4\int_A  \nabla\rho(x)\cdot X(x) + \rho(x)\dive X(x)\,dx \\
						 &= -4 \int_A \dive(\rho(x)X(x))\,dx \\
						 &= -4 \int_{\pt A} \rho(x)(X(x)\cdot\nu(x))\,d\Hn(x).
	 	\end{aligned}
	 	\nonumber
	\eeqn
Combining this with \eqref{eqn:first_var_per} we get that
	\beqn\label{eqn:weakfirst}
		\int_{\pt A} \big[(n-1)H(x)-4\sigma \rho(x)\big]\zeta(x)\,d\Hn(x) = 0,
	\eeqn
i.e., there exists a constant $\lambda$ such that
	\beqn\label{eqn:firstvar2}
		(n-1)H(x) - 4\sigma \rho(x) = \lambda
	\eeqn
for all $x\in\pt A$.
\end{proof}

\bigskip

\begin{remark}\label{rem:piecewise}
Proposition \ref{prop:firstsecondvar} holds locally in case $\rho$ is piecewise $C^1$.  That is, if we assume $\rho$ is $C^1$ except on a smooth submanifold (on which it or its derivative is allowed to jump), the curvature condition \eqref{eqn:firstvar} holds at regular points of $\rho$.  This observation follows by noting that the weak form \eqref{eqn:weakfirst} continues to hold for $\zeta$ supported in each component of the set of regular points of $\rho$, and that by appropriate choices of $\zeta$ we may conclude that the Lagrange multiplier $\lambda$ is independent of the component.
\end{remark}

\begin{remark}
Since the boundary of $A=\{x\colon u(x)=1\}$ of an $L^1$-local minimizer $u$ of the energy $\E_{\mu,\sigma}$ is of class $C^{1,\alpha}$ by Proposition \ref{prop:reg}, we can express the reduced boundary $\pt^* A$ locally as the graph of a $C^{1,\alpha}$ function $\varphi$ on a ball $B\subset\mathbb{T}^{n-1}$. Then, the first variation \eqref{eqn:firstvar2} implies that
	\[
		(n-1)\,\nabla\cdot\left(\frac{\nabla\varphi(x^\pr)}{\sqrt{1+|\nabla\varphi(x^\pr)|^2}}\right) = 4\sigma\rho(x^\pr,\varphi(x^\pr)) + \lambda \quad\text{ for }x^\pr\in B.
	\]
As the right-hand side is of class $C^1$, by standard elliptic regularity we obtain that $\varphi\in C^{3,\alpha}$. Hence, the boundary of $A=\{x\colon u(x)=1\}$ is of class $C^{3,\alpha}$ for some $\alpha>0$.
\end{remark}

\begin{remark}
Note that the condition \eqref{eqn:firstvar} is a sufficient condition for $u$ to be a critical point of the energy $\E_{\mu,\sigma}$ with respect to $L^1$-perturbations.
\end{remark}

\begin{remark}[Second Variation]
If we further assume that $\rho\in C^2(\Tn)$, then a necessary condition for $L^1$-local minimimality of $u$ is given via the second variation of the energy $\E_{\mu,\sigma}$ around the critical point. Namely,
		\beqn\label{eqn:secondvar}
			\int_{\pt A} \left( |\nabla_{\pt A}\zeta|^2-\|B_{\pt A}\|^2\zeta^2\right)\,d\Hn - 4\sigma\int_{\pt A} \left(\nabla \rho \cdot \nu \right)\zeta^2\,d\Hn(x) \geq 0
		\eeqn
for any smooth $\zeta:\pt A\arrow \R$ satisfying $\int_{\pt A}\zeta\,d\Hn=0$. Here $\nabla_{\pt A}\zeta$ denotes the gradient of $\zeta$ relative to the manifold $\pt A$, $B_{\pt A}$ denotes the second fundamental form of $\pt A$ and $\nu$ denotes the unit normal to $\pt A$ pointing out of $A$. The computation of \eqref{eqn:secondvar} follows by adapting the calculations in \cite[Theorem 2.6]{CS2}. 

In the absence of nanoparticles (when $\sigma=0$) an important result regarding the local minimizers of the nonlocal isoperimetric problem related to the energy \eqref{eqn:limit_energy_gamma} is given in \cite{AcFuMo13}. Here the authors prove that strict stability in the sense of positive definite second variation of critical sets is a sufficient condition of isolated local minimality with respect to the $L^1$-topology. We believe that the techniques introduced in \cite{AcFuMo13} can be adapted for the functional $\E_{\mu,\sigma}$ to conclude that strict positivity of \eqref{eqn:secondvar} implies local minimality in $L^1$.
\end{remark}

\section{An Example in Two Dimensions}\label{sec:T2}

Depending on the distribution of nanoparticles and the strength of penalization via $\sigma$ one can modify the phase morphology of block copolymers and effectively prescribe the location and shape of the phase transitions.  Even with a given measure $\mu\in\Prb_{\text{ac}}(\Tn)$ describing the particle distribution there are many possible critical patterns for the energy $\E_{\mu,\sigma}$, depending on the strength coefficient $\sigma>0$. Indeed, the penalization term can act as an \emph{attractive} or \emph{repulsive} term depending on the choice of $\mu$ via its pinning-like quality. The rigidity of the results in Propositions \ref{prop:reg} and \ref{prop:firstsecondvar}, on the other hand, limits the possibilities for critical and minimizing patterns. In this section we will provide such an example in two dimensions, i.e., on the 2-flat torus $\Ttwo$. Exploiting these rigidities, the example below shows that for a certain choice of $\mu$ and when the mass constraint $m$ is restricted to a certain range, for any $\sigma>0$ the global minimizer of the energy $\E_{\mu,\sigma}$ is geometrically quite different than the solution of the isoperimetric problem, i.e., when $\sigma=0$.

\bexmpl\label{exmpl:diskpen}
For $\Ttwo=[-1/2,1,2)\times[-1/2,1/2)$ with periodic boundary conditions, any $m\in [0,1-2/\pi)$, and any fixed $r>\sqrt{(2\pi)^{-1}(1+m)}$ let $\mu\in\Prb_{\text{ac}}(\Ttwo)$ be defined via the density function
	\[
		\rho(x):=\frac{1}{\pi r^2}\, \chi_{B(0,r)}(x)
	\]
for $x\in\Ttwo$.  

By the direct method in the calculus of variations, there exists a global minimizer of $\E_{\mu,\sigma}$ for any $\sigma>0$ and for the measure $\mu$ defined as above.  Let also $A$ denote the set $\{x\in\Ttwo\colon u_0(x)=1\}$. Since $\int_{\Ttwo}u_0(x)\,dx=m$, we have that $|A|=(1+m)/2$ and $|A^c|=(1-m)/2$.
The problem thus reduces to find a set $A\subset \Ttwo$ with area $|A|=(1+m)/2$ which minimizes
	\[
  		\E_{\mu,\sigma}(A) = \Per_{\Ttwo}(A) + 4\sigma\left( 1- {|A\cap B(0,r)|\over \pi r^2}  \right).
	\]
That is, $A$ should have as small a perimeter as possible, while maximizing its intersection with the nanoparticle domain $B(0,r)$.  We note that for $m\in[0,1-2/\pi)$, $(1-m)/2<(1+m)/2$ and by the choice of $r$ as above, we have that $|B(0,r)|>(1+m)/2$; hence, $A \cap B(0,r) \neq \emptyset$ for any admissible set $A$.

When $\sigma=0$ the material is nanoparticle-free, and the energy reduces to the classical isoperimetric problem on $\Ttwo$.  For our choice of mass $m\in[0,1-2/\pi)$, the unique minimizing configuration with $\sigma=0$ (up to translation) is the single striped lamellar pattern
	\[
		u_L(x_1,x_2)=\begin{cases}
						1  &\;\, \text{if } x_1 \in     \left(\frac{-1-m}{4},\frac{1+m}{4}\right), \\
						-1 &\;\, \text{if } x_1 \not\in \left(\frac{-1-m}{4},\frac{1+m}{4}\right),
					 \end{cases}
	\]
with associated set $A_L:=\{ x_1\in \left(\frac{-1-m}{4},\frac{1+m}{4}\right)\}$.
(See Figure \ref{fig:frst} in the Introduction.)  As noted above, by the choice of radius $r$, any translate of $A_L$ must intersect the nanoparticle site $B(0,r)$; as we will see below, this will imply that the lamellar pattern {\em cannot} be the energy minimizer for any $\sigma>0$, and in fact it will no longer be critical for the energy $\E_{\mu,\sigma}$.

\medskip

The shape of the minimizer is constrained by the curvature equations \eqref{eqn:firstvar} which are satisfied by any critical configuration.  Indeed, since the penalization density $\rho$ is uniformly bounded, by Proposition \ref{prop:reg}, we conclude that $\pt A$ is of class $C^{1,\alpha}$ for some $\alpha>0$. On the other hand, as $\rho$ is constant on each of $\inte(A\cap B(0,r))$ and $\inte(A\cap B^c(0,r))$, where $\inte$ denotes the interior of these sets, it is trivially differentiable, and the formula \eqref{eqn:firstvar} is locally valid (see Remark~\ref{rem:piecewise}.) Thus
	\beqn\label{eqn:curv_2D}	
		\begin{gathered}
			H(x)=\lambda \quad\text{for } x\in\pt A\cap \inte(B^c(0,r)),\text{ and} \\
	    	H(x)=\frac{2\sigma}{\pi r^2}+\lambda\quad \text{for } x\in\pt A\cap \inte(B(0,r))	
		\end{gathered}
	\eeqn
for some constant $\lambda$.  
We note that $H(x)$ denotes the signed curvature, and it is piecewise constant.  In particular, when $H>0$ the domain $A$ lies inside a circle of radius $1/H$, while for $H<0$, $A$ is exterior to a circle of radius $1/|H|$.

We may immediately confirm the claim made above, that the lamellar configurations, consisting of translations of $u_L$, cannot be minimizers (or even critical) for any $\sigma>0$.  Indeed, by our choices of parameters $m,r$, any translation of $A_L$ intersects $\inte B(0,r)$, so the curvature condition is violated in the intersection.  

Another observation which follows directly from the curvature conditions \eqref{eqn:curv_2D} is that any ball $A_R=B(p,R)$, with $R:=\sqrt{(2\pi)^{-1}(1+m)}$ and $p\in\Ttwo$ such that $A_R\subset B(0,r)$, is stationary for $\E_{\mu,\sigma}$.  (See Figure \ref{fig:thrd} below.)  This configuration is a local minimizer for any $\sigma\geq 0$, and in fact it is the global minimizer for all sufficiently large $\sigma$:

\begin{proposition}\label{prop:large}
There exists $\sigma_0=\sigma_0(m,r)$ such that for all $\sigma>\sigma_0$, $A_R$ (defined above) is a global minimizer of $\E_{\mu,\sigma}$.
\end{proposition}
We defer the proof of Proposition~\ref{prop:large} to the end of the section.

The question is then what is the geometry of minimizers for small positive values of $\sigma$. As the energy depends continuously on the parameter $\sigma>0$, for small values of $\sigma>0$ we expect that the global minimizer of the energy $\E_{\mu,\sigma}$ is $L^1$-close to a lamellar pattern when $0\leq m < 1-2/\pi$ (Figure \ref{fig:scnd}).  Below we propose a possible geometry for minimizers for small $\sigma$.  Although we cannot describe them completely, a minimizer which is not a disk must be ``stripe-like'' in the sense that it must exploit the topology of  $\Ttwo$:  

\bprop\label{prop:contract}  Let $m\in [0, 1-2/\pi]$ and $r>\sqrt{(2\pi)^{-1}(1+m)}$, and $A\subset\Ttwo$ corresponding to a minimizer of $\E_{\mu,\sigma}$.  Then:
\begin{enumerate}
\item If $A$ is contractible in $\Ttwo$, then $A$ is a ball of radius $R=\sqrt{(2\pi)^{-1}(1+m)}$ with $A\subset B(0,r)$.
\item $A^c$ cannot be contractible in $\Ttwo$.
\end{enumerate}
\eprop

\begin{proof}
First assume $A\subset\Ttwo$ is contractible.  We lift $\Ttwo$ to $\R^2$, its universal cover.  Contractibility in the torus implies that the lifting of $A$ consists of a periodic array of disjoint compact components $\tilde A\subset\R^2$, each with area $|\tilde A|=(1+m)/2$.  By the classical isoperimetric inequality, each component has perimeter $\Per_{\Ttwo}(\partial\tilde A)\geq \Per_{\Ttwo}(B_R)$, with equality if and only if the components are disks of radius $R$.  By placing a periodic array of disks of radius $R$ inside the array of translates of the nanoparticle site $B(0,r)$, we obtain a configuration which has smaller perimeter and which optimizes the penalization term, and thus has smaller energy than $A$, unless $A$ were also a disk of radius $R$ contained in $B(0,r)$.  Thus (i) is verified.

The case of $A^c\subset\Ttwo$ contractible is similar.  Since $A^c$ lifts to a periodic array of compact sets in $\R^2$, and $\Per_{\Ttwo}(\partial A^c)=\Per_{\Ttwo}(\partial A)$, we may conclude that a disk of radius $R$ again has smaller perimeter than $A$.  By locating the disk inside $B(0,r)$ the penalization term in $\E_{\mu,\sigma}$ is optimized, so again the disk has strictly smaller energy than any domain with $A^c$ contractible.  This proves (ii).
\end{proof}

\medskip

Using Proposition~\ref{prop:contract} and the curvature condition \eqref{eqn:curv_2D} we may illustrate some configurations which are candidates for the minimizer, and eliminate certain others.  Supposing that the minimizer is {\em not} a disk inside $B(0,r)$, we may assume that both $A$ and $A^c$ are not contractible, and hence each intersects both $B(0,r)$ and $B^c(0,r)$.   By the criticality conditions \eqref{eqn:curv_2D} we see that $\pt A$ has to be a union of arcs of circles and straight lines as its connected components have constant curvature in two dimensions. Also, note that the curvature of $\pt A$ inside the ball $B(0,r)$ has to be greater than the curvature of $\pt A$ on $B^c(0,r)$. Thus, $\pt A$ does not consist of a union of straight lines inside $B(0,r)$ and arcs of circles outside of $B(0,r)$.

Since $\pt A$ is of class $C^{1,\alpha}$ constant curvature components of $\pt A$ meet tangentially on $\pt B(0,r)$. Therefore, $\pt A$ can not consist of a union of an arc of a circle inside $B(0,r)$ connecting to another arc of a positively curved circle outside of $B(0,r)$, since two points and tangents at those points determine a circle uniquely, and two circles with different positive curvatures cannot meet at two points tangentially. Therefore the Lagrange multiplier $\lambda$ in \eqref{eqn:curv_2D} cannot be strictly positive since for $\lambda>0$ the components of $\pt A$ would consist of positively curved arcs of circles which is not possible. Therefore we may assume that either $\lambda=0$ or $\lambda<0$.

\medskip

\noindent {\bf Band aid patterns.} \ Suppose first that the Lagrange multiplier $\lambda=0$ in \eqref{eqn:curv_2D}. In this case the domain consists of arcs of circles inside $B(0,r)$ and straight lines outside of $B(0,r)$. Note that, by periodicity of the domain, the straight components of $\pt A$ in $B^c(0,r)$ must be parallel, and hence they must meet $\pt B(0,r)$ at semicircles inside $B(0,r)$. Such patterns we will refer to as \emph{band aid} patterns (see Figure \ref{fig:bandaids}). 

\begin{figure}[ht!]
     \begin{center}

        \subfigure[{\tiny One band aid}]{
            \label{fig:first}
            \includegraphics[width=0.25\linewidth]{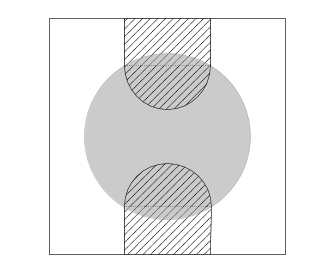}
        }\quad\
        \subfigure[{\tiny Two band aids}]{
           \label{fig:second}
           \includegraphics[width=0.25\linewidth]{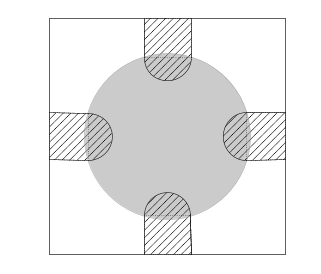}
        }\quad\ 
        \subfigure[{\tiny Slant band aid}]{
            \label{fig:third}
            \includegraphics[width=0.25\linewidth]{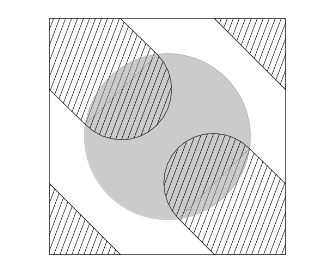}
        }

    \end{center}
    \caption{
        \emph{Band aid} patterns are stationary, but are {\bf not} global minimizers for $m\in[0,1-2/\pi)$ for any $\sigma>0$. Here again the gray disk depicts the penalization region $B(0,r)$.
     }
   \label{fig:bandaids}
\end{figure}
By adjusting the radii of the semicircles, we may match the area constraint and so these do represent stationary points of the energy $\E_{\mu,\sigma}$.
However, as the bandaid patterns are all contractible in $\Ttwo$, by Proposition~\ref{prop:contract} they can not be minimizers for any $\sigma>0$, and thus $\lambda=0$ is not achievable for a minimizer.

\bigskip

\noindent {\bf Concave/convex strips.} Suppose the Lagrange multiplier $\lambda<0$ in \eqref{eqn:curv_2D}.  Then $A$ lies inside of arcs of circles of radius $R_2$ inside of $B(0,r)$, and outside of circular arcs with radius $R_1$ outside of $B(0,r)$.  (See figure~\ref{fig:almoststripe}.)  Moreover, the curvature condition \eqref{eqn:curv_2D} relates the radii to the parameter $\sigma$ via
\begin{equation}\label{eqn:radii}
\frac{2\sigma}{\pi r^2} = \frac{1}{R_1} + \frac{1}{R_2}.
\end{equation}

	\begin{figure}[ht!]
		\begin{center}
			\includegraphics[height=5.5cm]{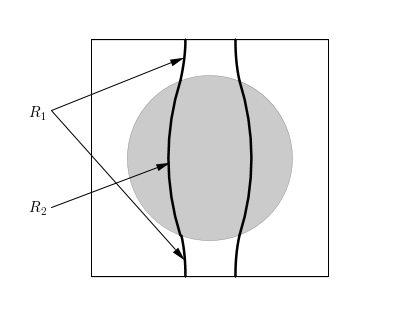}
		\end{center}
		\caption{
        The set $A$ is enclosed by arcs of circles where $\pt A$ is negatively curved outside of $B(0,r)$ and positively curved inside $B(0,r)$. The penalization region $B(0,r)$ is highlighted by the gray disk.
     }
   		\label{fig:almoststripe}
	\end{figure}

\medskip

\begin{lemma}\label{lem:R1}
 If $A$ is a minimizer, then $R_1>\frac12 -r$. 
 \end{lemma}

\begin{proof}
To verify the lemma, assume instead that $R_1\leq \frac12-r$.  The connected components of $A^c\setminus B(0,r)$ are then either contained inside circles of radius $R_1$ which are disjoint from $B(0,r)$, or are bounded by disjoint arcs of this radius which connect to $B(0,r)$ at two points on $\partial B(0,r)$.  In the latter case, we note that by lifting to $\R^2$, the distance between adjacent images of the nanoparticle domain $B(0,r)$ is $2(\frac12-r)>2R_1$.  Thus, the images in $\R^2$ of the components of $A^c$ are compact, and hence $A^c$ is contractible in $\Ttwo$.
This contradicts Proposition~\ref{prop:contract}, so therefore the lemma must hold true. 
\end{proof}

\medskip

We may now present our guess for stripe-like minimizers when $\sigma>0$ but small.  We choose the inside radius $R_2$ so as to create a concave/convex stripe pattern as in Figure~\ref{fig:almoststripe}.  In order to do this, it is necessary that the inside radius $R_2>r$; otherwise, the circular arcs inside $B(0,r)$ may not connect across the nanoparticle zone, and a curvy band aid pattern would result (see figure~\ref{fig:curvybandaid}.)  As such a pattern is contractible in $\Ttwo$, it cannot be a minimizer.  The concave/convex stripe thus requires a lower bound on both $R_1,R_2$, and hence can only be realized for
$$  \sigma = {\pi r^2\over 2}\left[\frac{1}{R_1} + \frac{1}{R_2}\right] <  {\pi r\over 2(1-2r)}.  $$
The exact values of $R_1,R_2$ (and the centers of the constructing circles) will be also determined by the area constraint $|A|=(m+1)/2$ and the requirement that the resulting curve is $C^{1,\alpha}$.

\begin{remark} We conjecture that when $\sigma>{\pi r\over 2(1-2r)}$ then the minimizer must be a disk of radius $R$, inside the nanoparticle region $B(0,r)$.  The variety of concave/convex regions which may be drawn is great (and is not restricted to shapes depicted in Figures~\ref{fig:almoststripe} and \ref{fig:curvybandaid},) so the optimum value of $\sigma_0$ in Proposition~\ref{prop:large} remains an open question.  However, we observe that as $\sigma$ gets larger, the radii $R_2$ of arcs within $B(0,r)$ must get smaller in order to satisfy \eqref{eqn:radii} (given Lemma~\ref{lem:R1}), and hence the area contained in the nanoparticle region is eventually insufficient to reduce the penalization term in the energy. (See Figure~\ref{fig:curvybandaid}.)  This observation forms the basis for our proof of Proposition~\ref{prop:large}.
\end{remark}

	\begin{figure}[ht!]
		\begin{center}
		\subfigure[{\tiny Convex/concave pattern with {\bf contractible} boundary }]{\label{curvy1}
			\includegraphics[width=0.22\linewidth]{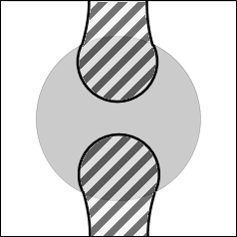}}\quad\quad\qquad\qquad
	    \subfigure[{\tiny Convex/concave pattern with {\bf uncontractible} boundary}]{\label{curvy2}
	        \includegraphics[width=0.22\linewidth]{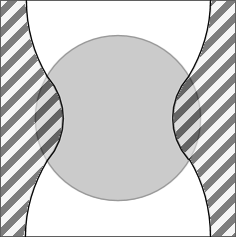}}
		\end{center}
		\caption{The sets $A$ shown are enclosed by arcs of circles where $\pt A$ is negatively curved outside of $B(0,r)$ and positively curved inside $B(0,r)$, with small radius $R_2<r$.  The left hand set yields a contractible region which cannot be a minimizer. The image on the right is unlikely to be a minimizer, as the penalization term will be large for $\sigma>0$.}
   		\label{fig:curvybandaid}
	\end{figure}

We conclude with the proof that the disk of radius $R$ gives the global minimizer for sufficiently large $\sigma$.
\begin{proof}[Proof of Proposition \ref{prop:large}.]
Let $A$ be the set associated to a global minimizer of $\E_{\mu,\sigma}$, and set $\beta:=R^2/4r^2<1/4$.  If $A$ is not a disk of radius $R$, then by Proposition~\ref{prop:contract} $A\cap B(0,r)^c\neq\emptyset$, and so $\partial A$ consists of arcs of circles of radius $R_1\ge 0$ outside $B(0,r)$, and of radius $R_2>0$ inside $B(0,r)$, satisfying \eqref{eqn:radii}.  
By Lemma~\ref{lem:R1}, we have
$$  {1\over R_2} > {2\sigma\over \pi r^2} - {2\over 1-2r}, $$
and thus there exists $\sigma_1=\sigma_1(r)$ so that for all $\sigma>\sigma_1$ we have
$R_2<\beta r/2$.

We now claim that  for all $\sigma>\sigma_1$, $A\cap B(0,r)$ lies inside a disjoint collection of circular arcs, each of which lies within distance $\beta r$ of $\partial B(0,r)$.  Indeed, $\partial A\cap B(0,r)$ consists of circular arcs of radius $R_2<\beta r/2$ (by the above estimate,) either connected to $\partial B(0,r)$ at the endpoints, or as disks of radius $R_2$ contained in the interior of $B(0,r)$.  
The arcs which contact $\partial B(0,r)$ lie within distance $\beta r$ of $\partial B(0,r)$ by the bound on $R_2$, so it remains to consider interior disks.  First,
assume that several such disks are contained in the interior of $B(0,r)$.  By the classical isoperimetric inequality, the perimeter of $A$ would be reduced by replacing these by a single disk with the same total area, with no change to the penalization term, and thus reducing the total energy. However, this contradicts the minimality of $A$, and thus there can only be a single disk of radius $R_2$ inside $B(0,r)$.  By translating this single disk to be tangent to $\partial B(0,r)$, the energy of $A$ remains the same, so we obtain a minimizer with all components of $A\cap B(0,r)$ within distance $\beta r$ of $\partial B(0,r)$, as claimed.

By the claim, $A\cap B(0,r)$ lies within an annular region $B(0,r)\setminus B(0,(1-\beta)r)$ of thickness $\beta r$.  In particular, 
$$|A\cap B(0,r)|< \pi (r^2 - [(1-\beta)r]^2) <2\pi r^2\beta. $$
We may then compare the energy of $A$ to that of the single disk $A_R\subset B(0,r)$ of radius $R=\sqrt{(2\pi)^{-1}(1+m)}$.  By the isoperimetric inequality on $\Ttwo$, $\Per_{\Ttwo}(A)\geq \Per_{\Ttwo}(A_L)= 2<\pi R^2$, given our choice of parameters.  Thus,
\begin{align*}
\E_{\mu,\sigma}(A) - \E_{\mu,\sigma}(A_R) & \geq
    2 - 2\pi R + 4\sigma\left( {R^2\over r^2} - {|A\cap B(0,r)|\over \pi r^2}\right) \\
   & \geq -\left[2\pi R -2\right] + 4\sigma\left( {R^2\over r^2} - 2\beta\right) \\
   & = -\left[2\pi R -2\right] + 2\sigma {R^2\over r^2} \\
   &>0,
\end{align*}
for all $\sigma>\max\{\sigma_1, {r^2\over R^2}(\pi R -1)\}:=\sigma_0(m,r)$.  Thus, for all $\sigma>\sigma_0$, the minimizer must be a disk contained inside $B(0,r)$.
\end{proof}
\medskip

\eexmpl

\medskip

\begin{remark}[Small translations] Note that unlike the function $u_L$, the function $u_R$ is unique only up to small translations, i.e., any translate $u_{R,a}$ of $u_R$ defined by
	\[
		u_{R,a}(x):=\begin{cases}
					1  &\;\, \text{if } x\in B(a,R), \\
							-1 &\;\, \text{if } x\not\in B(a,R),
				 \end{cases}
	\]
for any $a\in\Ttwo$ with $|a|<r-R$, we have that $\E_{\mu,\sigma}(u_{R,a})=\E_{\mu,\sigma}(u_R)$ if $R<r$. The energy $\E_{\mu,\sigma}$, though, is not translational invariant in general. This also reflects the ``pinning'' effect of the penalizing measure $\mu$.
\end{remark}

\begin{remark} We believe that the results of Example \ref{exmpl:diskpen} (Propositions \ref{prop:large} and \ref{prop:contract}) can be generalized easily to the case when the penalization measure is given by an indicator function $\rho$ satisfying (i) $\int_{\Tn}\rho(x)\,dx=1,$ (ii) $B(p,R) \subset\!\subset \supp\rho$ for some $p\in\Tn$ where $R=\sqrt{(2\pi)^{-1}(1+m)}$, and (iii) $|\supp\rho|>(1+m)/2$.
\end{remark}

\begin{remark}[The effect of $\sigma$] The effect of the penalization term in $\E_{\mu,\sigma}$ is more rigid than the effect of the nonlocal perturbation in $\E_{\mu,\sigma,\gamma}$ given by \eqref{eqn:limit_energy_gamma}. Indeed, in \cite{ST}, the authors show that on $\Ttwo$ the global minimizer of the nonlocal isoperimetric problem ($\E_{\mu,\sigma,\gamma}$ with $\sigma=0$), agrees with the global minimizer of the isoperimetric problem, i.e., is given by $u_L$, provided $\gamma>0$ is small. That is, the perimeter term dominates and the effect of the nonlocal perturbation via $\gamma>0$ does not ``kick-in'' immediately whereas the above example shows that this is not the case for $\sigma>0$.
\end{remark}

\section{Concluding Remarks}\label{sec:conc_rem}

As noted in the introduction and as the example in Section \ref{sec:T2} shows perhaps the most important feature of minimizing the energy $\E_{\mu,\sigma}$ is that compared to the isoperimetric problem the geometry of minimizing patterns can change significantly. Via its connection to the energy $\E_{\eps,\sigma,r,N}$ (Section \ref{sec:infin_part}), this reflects well the physical applications of adding nanoparticles into copolymer blends to change the morphology of pattern formation. Indeed, since the consideration of the energy \eqref{eqn:limit_energy} is to our knowledge the first mathematically rigorous study of nanoparticle/copolymer blends, this work also generates several directions for subjects of future studies. We will conclude by remarking on these directions.

 \begin{enumerate}[(1)]
 
	\item As mentioned before, depending on the choice of the penalizing measure $\mu$, the second term in $\E_{\mu,\sigma}$ can act as an \emph{attractive} or \emph{repulsive} term. For example, in two dimensions and small mass regime, by choosing the measure $\mu$ distributed on disjoint small disks one can force the minimizer of the energy $\E_{\mu,\sigma}$ to ``oscillate'' rather than forming a larger disk which would be preferable in terms of minimizing the perimeter term (see Figure \ref{fig:dots}).
	\begin{figure}[ht!]
		\begin{center}
			\includegraphics[height=4cm]{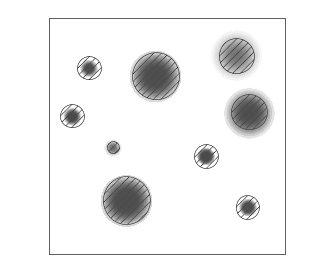}
		\end{center}
		\caption{
        A measure $\mu$ consisting of small disjoint supports (light gray blobs) might show a repulsive effect when a minimizer (striped regions) try to cover most of the support of $\mu$ to reduce cost.
     }
   		\label{fig:dots}
	\end{figure}
	
	\item Although the choice of the measure $\mu$ provides a substantial freedom in forcing the minimizers of $\E_{\mu,\sigma}$ to form desired patterns, the minimizing patterns still exhibit some rigidity. In particular, regularity properties (Proposition \ref{prop:reg}) and the criticality and stability conditions (Proposition \ref{prop:firstsecondvar}) limit this freedom (see also the example in Section \ref{sec:T2}). Adding the long-range interaction term between the phases controlled by $\gamma$ as in Remark \ref{rem:nonloc_pert} would enrich the possibilities for minimizing patterns. This would also provide further mathematical challenges in understanding the energy landscape of \eqref{eqn:limit_energy_gamma}.
	
	\item Here we chose to fix the location of nanoparticles, hence, their distribution given by the measure $\mu$ as the number of particles goes to infinity whereas their size approach zero. An interesting problem would be to analyze local and global minimizers of the energy $\E_{\mu,\sigma}$ not only with respect to the phases, i.e., over $u\in BV(\Tn;\{\pm 1\})$ with a fixed mass constraint $m$ and fixed measure $\mu$, but also over the measures $\mu\in\Prb_{\text{ac}}(\Tn)$.

 \end{enumerate}

\bigskip

\noindent {\bf Acknowledgements.} The authors were supported by NSERC (Canada) Discovery Grants. IT was also supported by a Field--Ontario Postdoctoral Fellowship. The authors would like to thank the anonymous reviewers for their comments.

\bibliographystyle{plain}
\bibliography{NanopartBib}

\end{document}